\documentclass[a4paper,11pt]{article}
\usepackage{latexsym}
\usepackage{amssymb}
\usepackage{enumerate}
\usepackage{enumitem}
\usepackage{xcolor}
\usepackage{commath}
\usepackage{array}
\usepackage{amsfonts}
\usepackage{amsmath,amsthm}
\usepackage{hyperref}
\usepackage{algorithm}
\usepackage{algorithmic}
\usepackage{framed}
\usepackage{boites}
\usepackage{graphicx}
\usepackage{cases}

\newenvironment{@abssec}[1]{%
    \if@twocolumn

      \section*{#1}%
    \else

      \vspace{.05in}\footnotesize
      \parindent .2in
 {\upshape\bfseries #1. }\ignorespaces
    \fi}

    {\if@twocolumn\else\par\vspace{.1in}\fi}

\newenvironment{keywords}{\begin{@abssec}{\keywordsname}}{\end{@abssec}}

\newenvironment{AMS}{\begin{@abssec}{\AMSname}}{\end{@abssec}}

\newcommand\keywordsname{Key words}
\newcommand\AMSname{AMS subject classifications}
\newcommand\AMname{AMS subject classification}
\newcommand\restr[2]{{
\left.\kern-\nulldelimiterspace 
#1 
\vphantom{|} 
\right|_{#2} 
}}
\newtheorem{theorem}{Theorem}[section]
\newtheorem{lemma}[theorem]{Lemma}
\newtheorem{corollary}[theorem]{Corollary}
\newtheorem{proposition}[theorem]{Proposition}
\newtheorem{remark}[theorem]{Remark}

\newtheorem{problem}{Problem}

\newtheorem{test}{Theorem}

\newtheorem{thm}{Theorem}

\newcommand{\RR}{\mathbb{R}}

\renewcommand{\SS}{\mathbb{S}}

\def\XXint#1#2#3{{\setbox0=\hbox{$#1{#2#3}{\int}$}
\vcenter{\hbox{$#2#3$}}\kern-.5\wd0}}

\newcommand{\supp}{\mathop{\mathrm{supp}}}

\newcommand{\dist}{\mathop{\mathrm{dist}}}  
\newcommand{\link}{\mathop{\circ\kern-.35em -}}

\newcommand{\ol}{\overline}
\newcommand{\pa}{\partial}

\newcommand{\dv}{\mathop{\mathrm{div}}}

\newcommand{\gr}{\nabla}
\newcommand{\sn}{<}

\newcommand{\al}{\alpha}
\newcommand{\be}{\beta}
\newcommand{\ga}{\gamma}  
\newcommand{\Ga}{\Gamma}
\newcommand{\de}{\delta}
\newcommand{\De}{\Delta}
\newcommand{\ve}{\varepsilon}
 
\newcommand{\la}{\lambda}
\newcommand{\La}{\Lambda}

\newcommand{\Si}{\Sigma}

\newcommand{\om}{\omega}
\newcommand{\Om}{\Omega}
\newcommand{\rn}{{\mathbb{R}}^N}
\newcommand{\sg}{\sigma}

\newcommand\setbld[2]{\left\{ #1 \; :\; #2\right\}}

\newcommand{\tin}{{\text{in }}}
\newcommand{\ton}{{\text{on }}}

\newcommand{\tfor}{{\text{for }}}
\newcommand{\id}{{\rm Id}}
\newcommand\pp[1]{\left( #1\right)}
\newcommand{\inv}{^{-1}}
\newcommand{\cdottone}{{\boldsymbol{\cdot}}}

\newcommand{\cA}{\mathcal{A}}
\newcommand{\C}{\mathcal{C}}
\newcommand{\cC}{\mathcal{C}}

\newcommand{\cF}{\mathcal{F}}
\newcommand{\cG}{{\mathcal G}}
\newcommand{\cH}{{\mathcal H}}

\newcommand{\cX}{\mathcal{X}}

\title{The simultaneous asymmetric perturbation method\\ for overdetermined free boundary problems
\thanks{This research was partially supported by the 
Grant-in-Aid for Research Activity Start-up (No. 20K22298) of the Japan Society for the Promotion of Science.}
}

\author{Lorenzo Cavallina 
}
\date{}

\begin{document}

\maketitle

\begin{abstract}
In this paper, we introduce a new method for applying the implicit function theorem to find nontrivial solutions to overdetermined problems with a fixed boundary (given) and a free boundary (to be determined). The novelty of this method lies in the kind of perturbations considered. Indeed, we work with perturbations that exhibit different levels of regularity on each boundary. This allows us to construct solutions that would have been out of reach otherwise. Another benefit of this method lies in the improvement of the regularity gap between the free boundary and the given one. Finally, some geometric properties of the solutions, such as symmetry and convexity, are also discussed.  
\end{abstract}



\pagestyle{plain}
\thispagestyle{plain}

\section{Introduction}\label{introduction}
\subsection{Problem setting and known results}
Let $\Om$ be a smooth bounded domain of $\rn$ ($N\ge2$) and $D\subset\ol D \subset \Om$ be a subdomain with Lipschitz continuous boundary $\pa D$. For simplicity, we will require that both boundaries $\pa D$ and $\pa\Om$ are connected. 
Moreover, let $n$ denote the outward unit normal vector to either  
$\pa\Om$ or $\pa D$ depending on the context.
In what follows, we consider various free boundary problems that satisfy the following assumptions: 
\begin{itemize}
\item One of the two boundaries (say $\pa D$) is given, while the other (say $\pa \Om$) is the free boundary to be determined.
\item The solution of a certain boundary value problem (depending on $\Om$ and $D$) satisfies a given overdetermined condition on the free boundary.
\item The pair $(\pa D_0,\pa \Om_0)$ is some known solution to the free boundary problem (\emph{trivial solution}).
\end{itemize}

Undoubtedly, one of the most famous examples of free boundary problems that satisfy the properties above is the following Bernoulli overdetermined problem.
\begin{problem}\label{pb 1}
Find a pair $(D,\Om)$ such that the following overdetermined problem admits a solution for some real parameter $c$.
\begin{equation}\label{bernu}
    \begin{cases}
    -\De u = 0 \quad \text{in }\Om\setminus \ol D,\\
    u=1 \quad \text{on }\pa D,\\
    u=0 \quad \text{on }\pa \Om,\\
    \pa_n u = c \quad \text{on }\pa \Om.
    \end{cases}
\end{equation}
Here $\pa_n$ stands for the outward normal derivative at the boundary.
\end{problem}
Clearly, \eqref{bernu} is solvable if $(D,\Om)$ is a pair of concentric balls (trivial solution). Existence, regularity and qualitative properties of the solutions of \eqref{bernu} have been studied for a long time and a plethora of different approaches is known. As far as the study of the properties of single solutions is concerned, we refer to \cite{beurling, altcaff, HS97} and the references therein. Furthermore, as far as the study of families of solutions is concerned, we refer to \cite{acker, henrot onodera} and the references therein. 

Another example of a free boundary problem that fits the description above is given by the two-phase overdetermined problem of Serrin type. 
Given a positive constant $\sg_c\ne 1$, define the following piece-wise constant function: 
\begin{equation}\label{sigma}
    \sg=\sg_c\ \cX_D+\cX_{\Om\setminus D},
\end{equation}
where $\cX_A$ is the characteristic function of the set $A$ (i.e., $\cX_A(x)$ is $1$ if $x\in A$ and $0$ otherwise).
\begin{problem}\label{pb 2}
Find a pair $(D,\Om)$ such that the following overdetermined problem admits a solution for some real parameter $c$.
\begin{equation}\label{2ph}
    \begin{cases}
    -\dv\left(\sg\gr u\right)=1 \quad \text{in }\Om, \\
    u=0\quad \text{on }\pa\Om,\\
    \pa_n u=c \quad \text{on }\pa\Om.
    \end{cases}
\end{equation}
\end{problem}

This problem was first studied by Serrin \cite{Se1971} in the particular case $D=\emptyset$. He showed that, if $D=\emptyset$, Problem \ref{pb 2} admits a solution if and only if $\Om$ is a ball. 
Just like Problem \ref{pb 1}, \eqref{2ph} is solvable if $(D,\Om)$ is a pair of concentric balls (trivial solution). Nontrivial solutions of Problem 2 have been studied only in recent years (see \cite{CY1, CYisaac} for the local behavior of nontrivial solutions near the trivial ones) and still very little is known about them in the general case. 

\subsection{The simultaneous asymmetric perturbation method}
In what follows, we will briefly describe the ideas behind the simultaneous asymmetric perturbation (SAP) method.
For simplicity let $(D_0,\Om_0)$ denote the pair of concentric balls centered at the origin with radii $R$ and $1$ respectively ($0\sn R\sn 1$). Now, let $\cF$ and $\cG$ be two suitable Banach spaces to be defined later. 
For sufficiently small $f\in\cF$ and $g\in\cG$, we introduce the perturbed domains $D_f$ and $\Om_g$ (see \eqref{D_f Om_g} for the precise definition)

\begin{figure}[h]
\centering
\includegraphics[width=0.45\linewidth]{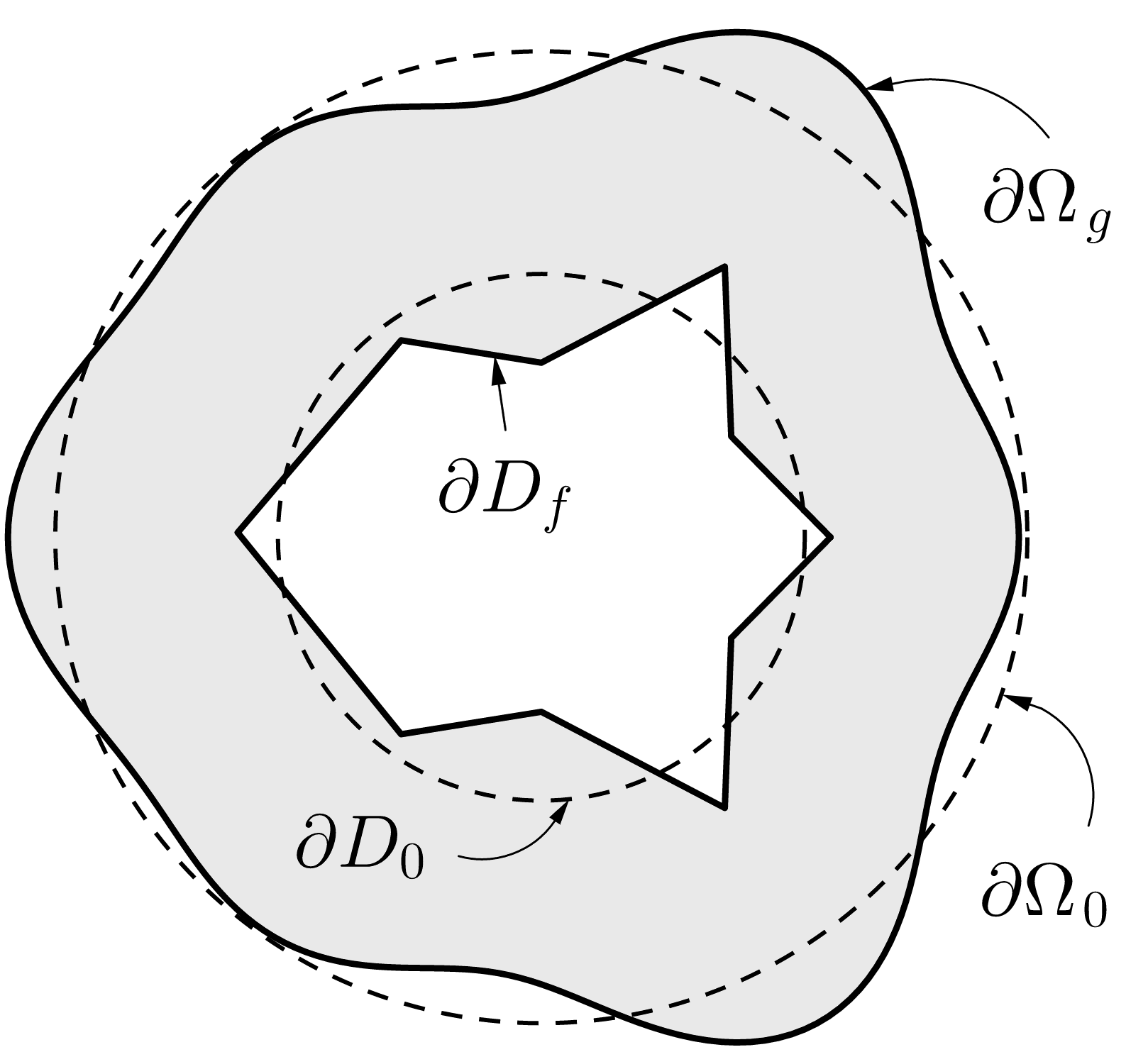}
\caption{Geometrical setting of the simultaneous asymmetric perturbation method.} 
\label{nonlinear}
\end{figure}

The first step to apply the SAP method consists in finding another Banach space $\cH$ and a mapping 
\begin{equation*}
    \Psi: \cF\times\cG\longrightarrow\cH
\end{equation*}
such that $\Psi(f,g)=0$ if and only if the pair $(D_f,\Om_g)$ solves the given overdetermined problem. The second step consists in applying the following version of the implicit function theorem (\cite{AP1983}) for Banach spaces to $\Psi$.  
\begin{thm}[Implicit function theorem]\label{ift}
Let $\Psi\in\cC^k(\La\times W,\cH)$, $k\ge1$, where $\cH$ is a Banach space and $\La$ (resp. $W$) is an open set of a Banach space $\cF$ (resp. $\cG$). Suppose that  $\Psi(f^*,g^*)=0$ and that the partial derivative $\pa_g\Psi(f^*,g^*)$ is a bounded invertible linear transformation from $\cG$ to $\cH$. 

Then there exist neighborhoods $\Theta$ of $f^*$ in $\cF$ and $W^*$ of $g^*$ in $\cG$, and a map $g\in\cC^k(\Theta,\cG)$ such that the following hold:
\begin{enumerate}[label=(\roman*)]
\item $\Psi(f,\widetilde g(f))=0$ for all $f\in\Theta$,
\item If $\Psi(f,g)=0$ for some $(f,g)\in\Theta\times W^*$, then $g=\widetilde g(f)$,
\item $(\widetilde g)'(f)=-[\pa_g \Psi(p) ]^{-1}\circ \pa_f \Psi(p)$, where $p=(f,\widetilde g(f))$ and $f\in\Theta$.
\end{enumerate}
\end{thm}

The essence of the SAP method relies upon the fine-tuned choice of the Banach spaces $\cF$, $\cG$ and $\cH$ so that
\begin{itemize}
    \item $\cF$ is as large as possible. That is, functions in $\cF$ enjoy the lowest regularity possible.
    \item $\cG$ is as small as possible. That is, functions in $\cG$ enjoy the highest regularity possible.
    \item The map $\Psi:\cF\times\cG\to\cH$ satisfies the hypotheses of the implicit function theorem in a neighborhood of $(0,0)\in\cF\times\cG$. In particular $\Psi$ has to be Fr\'echet differentiable \textbf{jointly} in the variables $f$ and $g$.
\end{itemize}

The most technical aspect of this method lies in showing the Fr\'echet differentiability of $\Psi$ and in performing the resulting computations. This can be done by employing the theory of shape derivatives, some mathematical machinery for computing derivatives of shape functionals (such as $\Psi$) with respect to geometric perturbations of the boundary (see \cite{SG, HP2018}). The novelty of this paper lies in the kind of perturbations considered. Indeed, we will examine the case of perturbations that are both \textbf{simultaneous} (in $f$ and $g$, that is, perturbing both $\pa D_0$ and $\pa \Om_0$ at the same time) and \textbf{asymmetric} (that is $f$ and $g$ enjoy different regularities). To the best of my knowledge, I am not aware of any other work in the literature where it is meaningful to consider shape derivatives with respect to simultaneous asymmetric perturbations of two boundaries. 

\subsection{Main results}
Let $B$ denote an open ball centered at the origin such that $\ol D_0\subset B\subset \ol B\subset \Om_0$. 
Moreover, for an integer $m\ge 2$ and a real number $\al\in (0,1)$, consider the following Banach spaces endowed with the usual norms (that will be simply denoted by $\norm{\cdottone}$):
\begin{equation}\label{FGH}
\begin{aligned}
    &\cF:=\left\{f\in\cC^{0,1}(B,\rn)\;:\; f\equiv 0 \quad \ton \pa B\right\},\quad 
    \cG:=\left\{g\in\cC^{m,\al}(\pa \Om_0, \RR)\;:\; \int_{\pa \Om_0} g =0\right\}, \\ 
    &\cH:=\left\{h\in\cC^{m-1,\al}(\pa \Om_0, \RR)\;:\; \int_{\pa \Om_0} h =0\right\}.
\end{aligned}
\end{equation}

Now, for small $(f,g)\in\cF\times\cG$ let $\varphi_{f,g}\in W^{1,\infty}(\rn,\rn)$ be a map such that ${\rm Id}+ \varphi_{f,g}:\rn\to\rn$ is a diffeomorphism (here $\rm Id$ denotes the identity mapping of $\rn$) and
\begin{equation}\label{extending perturbations}
    \varphi_{f,g}= f \;\tin B, \quad \varphi_{f,g}=gn \;\ton \pa\Om_0.   
    \end{equation}
Finally, for $(f,g)\in\cF\times\cG$ sufficiently small, set
\begin{equation}\label{D_f Om_g}
D_f:=({\rm Id}+\varphi_{f,g})(D_0), \quad 
\Om_g:=  ({\rm Id}+\varphi_{f,g})(\Om_0).
\end{equation}

\begin{test}\label{main thm 1}
There exists a threshold $\ve>0$ such that, for all $f\in\cF$ satisfying $\norm{f}<\ve$ there exists a function $g=g(f)\in\cG$ such that the pair $(D_f,\Om_g)$ is a solution to problem \eqref{bernu} for some $c\in\RR$. Moreover, this solution is unique in a small enough neighborhood of $(0,0)\in\cF\times\cG$.
\end{test}
Moreover, the asymptotic behavior of the function $g(f)$ above as $\norm{f}\to0$ is given by Corollary \ref{asymptotic thm 1}.

Let us define
\begin{equation*}
\begin{aligned}
s(k)&= \frac{k(N+k-1)-(N+k-2)(k-1)R^{2-N-2k}}{k( N+k-1)+k(k-1)R^{2-N-2k}} \text{ for }k = 1,2,\ldots,\\
\Sigma&=\setbld{s\in (0,\infty)}{s=s(k)
\,\text{ for some }k = 1,2,\ldots}.
\end{aligned}
\end{equation*}
\begin{test}\label{main thm 2}
Let $\sg_c\in(0,\infty)\setminus\Sigma$. Then, there exists a threshold $\ve>0$ such that, for all $f\in\cF$ satisfying $\norm{f}_{\cC^{0,1}}<\ve$, there exists a function $g=g(f,\sg_c)\in\cG$ such that the pair $(D_f,\Om_g)$ is a solution to problem \eqref{2ph} for some $c\in\RR$. Moreover, this solution is unique in a small enough neighborhood of $(0,0)\in\cF\times\cG$.
\end{test}
Moreover, the asymptotic behavior of the function $g(f)$ as $\norm{f}_{\cC^{0,1}}\to0$ is given by Corollary \ref{asymptotic thm 2}.

The local behavior of the solutions of Problem 2 when $\sg_c\in\Si$ has been carried out in \cite{CYisaac}. 

This paper is organized as follows. In Section 2 we give the basic definitions concerning shape derivatives. In Section 3 we prove Theorem \ref{main thm 1} through the implicit function theorem (Theorem \ref{ift}). Similarly, Section 4 is devoted to the proof of Theorem \ref{main thm 2}. In Section 5 we give some general remarks on some geometric properties of the solutions $(D_f,\Om_g)$ (namely regularity, symmetry and convexity). Finally, in the Appendix, we prove a technical lemma that is crucial for the SAP method.

\section{Preliminaries on shape derivatives}

In this section, we will introduce the concept of shape derivatives. 
Let us first introduce some basic notation. Let $\om\subset\rn$ be a smooth domain at which we will compute the derivative of a shape functional $J$ (to this end, we will require $J(\widetilde{\om})$ to be defined at least for all domains $\widetilde{\om}$ ``sufficiently close'' to the reference domain $\om$).
Let $\varphi_0:\ol \omega\to\rn$ be a sufficiently ``smooth" vector field. Let $\omega_t:=({\rm Id}+t \varphi_0)(\om)$. For $t>0$ small enough the perturbation of the identity ${\rm Id}+t\varphi_0:\ol\om\to\ol{\om_t}$ is a diffeomorphism. 
The shape derivative of $J$ at $\om$ with respect to the perturbation field $\varphi_0$ is then defined as
\begin{equation*}
J'(\om)(\varphi_0)=\lim_{t\to 0} \frac{J(\om_t)-J(\om)}{t}.
\end{equation*}
Of course, the definition above can be extended to functionals that take several domains as input as well.

The concept of shape derivative can be applied to shape functionals that take values in a general Banach space too. 
A fairly common example is given by a smoothly varying family of sufficiently ``smooth" real-valued functions $w_t$ defined on the set $\om_t$ (in many practical applications $w_t$ is the solution to some boundary value problem defined on the perturbed domain $\om_t$). Since each $w_t$ lives in a different domain $\om_t$, the shape derivative $w'$ has to be defined in an indirect way (see \cite{SG}), that is
\begin{equation}\label{f' by material deri}
w'=\dot w - \gr w\cdot \varphi_0,     
\end{equation}
where $\dot w$ is the so-called \emph{material derivative} of $w_t$, defined as
\begin{equation*}
\dot w= \restr{\frac{d}{dt}}{t=0} w_t\circ \left({\rm Id} + t\varphi_0\right).    
\end{equation*}
We remark that, under suitable regularity conditions, the value of the shape derivative $w'(x)$ at any point $x\in\om$ simply coincides with the derivative $\restr{\frac{d}{dt}}{t=0} w_t(x)$ (indeed, notice that $x\in\om_t$ for $t$ small enough).

\section{Proof of Theorem \ref{main thm 1}}
\subsection{Preliminaries}

Let $\cF$, $\cG$ be the Banach spaces defined in \eqref{FGH}. For $(f,g)\in\cF\times\cG$, let $u_{f,g}$ denote the solution to the following boundary value problem:
\begin{equation}\label{dirichlet bernu}
\begin{cases}
-\De u=0 \;\; \tin \Om_g\setminus \ol {D_f}, \\
u=1 \;\; \ton \pa D_f, \\
u=0 \;\; \ton \pa\Om_g.
\end{cases}
\end{equation}
An elementary calculation yields that, for $(f,g)=(0,0)$, the solution of \eqref{dirichlet bernu} is given by the following radial function: 
\begin{equation}\label{bernu0}
u(r)=  \begin{cases}\displaystyle
\frac{r^{2-N}-1}{R^{2-N}-1} \quad \text{for }N\ge3,\\[10pt]
\displaystyle \frac{\log r}{\log R} \quad \text{for }N=2.
\end{cases}
\end{equation}
Moreover, by the above, we have
\begin{equation}\label{normal derivatives on the boundary}
\pa_n u \equiv \frac{2-N}{R^{2-N}-1},\qquad \pa_{nn}u \equiv \frac{2-N}{R^{2-N}-1} \ (1-N) \quad \ton \pa\Om_0,
\end{equation}
where we employed the following convention:
\begin{equation}\label{convention}
\frac{2-N}{R^{2-N}-1}:= \frac{1}{\log R} \quad \tfor N=2.    
\end{equation}
As announced in the introduction, we will apply the implicit function theorem to the function:
\begin{equation}\label{Psi}
    \begin{aligned}
    \Psi: & \cF\times \cG \longrightarrow \cH\\
    &(f,g)\longmapsto \Pi_0 \left( \restr{\left(\pa_{n_g} u_{f,g}\right)}{\pa\Om_0}\right),
\end{aligned}
\end{equation}
where $\Pi_0:\cC^{m-1,\al}(\pa\Om_0)\to\cH$ is the projection operator defined by
\begin{equation*}
\Pi_0(\varphi)=\varphi-\frac{1}{|\pa\Om_0|}\int_{\pa\Om_0}\varphi
\end{equation*}
and $u_{f,g}$ is the solution of \eqref{dirichlet bernu}. Moreover, by a slight abuse of notation, here $\restr{\left(\pa_{n_g} u_{f,g}\right)}{\pa\Om_0}$ denotes the function of value
\begin{equation}\label{function of value}
\gr u_{f,g}\big(x+g(x)n(x)\big)\cdot   n_g\big(x+g(x)n(x)\big) \quad \tfor x\in\pa\Om_0.  \end{equation}
Finally, notice that, by the Hopf lemma and the boundary condition in \eqref{bernu},  we can rewrite \eqref{function of value} as 
\begin{equation}\label{function of value rewritten}
\restr{\left(\pa_{n_g} u_{f,g}\right)}{\pa\Om_0}(x) = -\left|\gr u_{f,g}\big(x+g(x)n(x)\big)\right| \quad \tfor x\in\pa\Om_0.    
\end{equation}

\subsection{Computing the shape derivative of $u_{f,g}$}
In this subsection, we will compute the explicit expression of the shape derivative of $u_{f,g}$ (the question of the shape differentiability of $u_{f,g}$ will be addressed by Remark \ref{shape differentiability u_fg} in the Appendix).
Let $\varphi_0:\ol{\Om_0}\to\rn$ satisfy \eqref{extending perturbations}.
Then, the shape derivative $u'$ can be characterized as the unique solution of the following boundary value problem (the proof is analogous to that of \cite[Theorem 5.3.1]{HP2018}).
\begin{equation}\label{bernu u' eq}
\begin{cases}
-\De u' = 0 \quad \text{in }\Om_0\setminus\ol D_0,\\
u'=-\pa_n u \ f_0\cdot n \quad \text{on } \pa D_0,\\
u'=-\pa_n u \ g_0 \quad \text{on } \pa \Om_0.
\end{cases}
\end{equation}
We remark that \eqref{bernu u' eq} depends on the perturbation field $\varphi_0$ only through its normal component $\varphi_0\cdot n$ on $\pa D_0\cup \pa\Om_0$ (this fact holds in general and is known as the \emph{structure theorem} for shape derivatives in the literature \cite{structure}).  
In what follows we will also make use of the following notation for partial shape derivatives.  Let $u'_-$ and $u'_+$ denote the solution to \eqref{bernu u' eq} corresponding to the pairs $(f_0, 0)$ and $(0,g_0)$ respectively. Notice that, by linearity, we have $u'=u'_-+u'_+$.

Let $\{Y_{k,i}\}_{k,i}$ ($k\in \{0,1,\dots\}$, $i\in\{1,2,\dots, d_k\}$) denote a maximal family of linearly independent solutions to the eigenvalue problem
\begin{equation}\label{sphardef}
-\De_{\tau} Y_{k,i}=\lambda_k Y_{k,i} \quad\textrm{ on }\SS^{N-1},
\end{equation}
were $\De_\tau$ stands for the Laplace--Beltrami operator on the unit sphere $\SS^{N-1}$. The $k$-th eigenvalue $\lambda_k=k(N+k-2)$ has multiplicity $d_k$. Moreover, we consider the normalization $\norm{Y_{k,i}}_{L^2(\SS^{N-1})}=1$.  The solutions to the eigenvalue problem above, usually referred to as \emph{spherical harmonics},  form a complete orthonormal system of $L^2(\SS^{N-1})$. Finally, notice that the eigenspace corresponding to the eigenvalue $\la_0=0$ is the 1-dimensional space of constant functions on $\SS^{N-1}$.

\begin{proposition}\label{prop expansion}
Let $(f_0,g_0)\in\cF\times\cG$ and assume that, for some real coefficients $\alpha_{k,i}^\pm$, the following expansions hold true in $L^2(\SS^{N-1})$:
\begin{equation}\label{h_in h_out exp}
f_0(R\theta)\cdot \theta =\sum_{k=0}^\infty\sum_{i=1}^{d_k}\al_{k,i}^- Y_{k,i}(\theta), \quad 
g_0(\theta)=\sum_{k=1}^\infty\sum_{i=1}^{d_k}\al_{k,i}^+ Y_{k,i}(\theta), \quad \theta\in\SS^{N-1}.
\end{equation} 
Then, the function $u'=u'_-+u'_+$, solution to \eqref{bernu u' eq}, admits the following explicit expression for $\theta\in\mathbb{S}^{N-1}$ and $r\in [R,1]$:
\begin{equation*}
    u'_\pm(r,\theta) = \displaystyle \sum_{k=0}^\infty\sum_{i=1}^{d_k} \left( A_k^\pm r^{2-N-k}+B_k^\pm r^k  \right) \al_k^\pm Y_{k,i}(\theta). 
\end{equation*}
The values of the coefficients $A_k^\pm$ and $B_k^\pm$ are given by 
\begin{equation*}
A_k^-=-B_k^-= 
\displaystyle \frac{2-N}{R^{2-N}-1}\frac{-R^{2-N}}{(R^{2-N-k}-R^k)}, 
\end{equation*}
\begin{equation*}
    A_k^+=\displaystyle\frac{-B_k^+}{R^{2-N-2k}}=
    \displaystyle\frac{2-N}{R^{2-N}-1}\frac{1}{(R^{2-N-2k}-1)},
\end{equation*}
where we made use of the convention \eqref{convention}
\end{proposition}
\begin{proof}
The proof is virtually identical to that of \cite[Section 4]{cava2018}. We will compute here the expression for $u'_+$ only, since the case of $u'_-$ is completely analogous. 
Let us pick arbitrary $k\in\{1,2,\dots\}$ and $i\in\{1,\dots, d_k\}$. We will use the method of separation of variables to find the solution of problem \eqref{bernu u' eq} in the particular case when $f_0=0$ on $\pa D_0$ and $g_0=Y_{k,i}$ on $\pa\Om_0$ and then the general case will be recovered by linearity.
We will be searching for solutions of the form $u'_+=u'_+(r,\theta)=S(r)Y_{k,i}(\theta)$ (where $r:=\abs{x}$ and $\theta:=x/\abs{x}$ for $x\ne 0$).
Using the well known decomposition formula for the Laplace operator into its radial and angular components (see for instance \cite[Proposition 5.4.12]{HP2018}), the equation $\De u'_+=0$ in $\Om_0\setminus\ol{D_0}$ can be rewritten as
$$
\pa_{rr}S(r)Y_{k,i}(\theta)+\frac{N-1}{r}\pa_r S(r)Y_{k,i}(\theta)+\frac{1}{r^2}S(r)\De_\tau Y_{k,i}(\theta)=0 \quad\text{for }r\in (R,1),\, \theta\in\SS^{N-1}.
$$
By \eqref{sphardef}, we get the following equation for $S$:
\begin{equation}\label{f}
\pa_{rr}S+\frac{N-1}{r}\pa_r S-\frac{\lambda_k}{r^2}S=0 \quad\text{in } (R,1).
\end{equation}
Since we know that $\la_k=k(k+N-2)$, it can be easily checked that any solution to the above consists of a linear combination of the following two independent solutions:
\begin{equation}\label{xieta}
S_{sing}(r):= r^{2-N-k}\ \quad \text{ and } \quad S_{reg}(r):= r^k. 
\end{equation}
Then, for some real constants $A_k^+$, $B_k^+$ we have
$$
S(r)= 
A_k^+r^{2-N-k}+B_k^+r^k \quad \text{for } r\in(R,1).
$$
The coefficients $A_k^+$ and $B_k^+$ can then be obtained by the boundary conditions of problem \eqref{bernu u' eq} by setting $f_0=0$. 
We get the following system:
\[
\begin{cases}
A_k^+ R^{2-N-2k} + B_k^+ =0, \\
A_k^+ + B_k^+ = -\frac{(2-N)}{R^{2-N}-1},
\end{cases}
\]
where we make use of the convention \eqref{convention} for $N=2$.
By solving it we obtain the coefficients of the series representation of $u'_+$.
\end{proof}
\subsection{Computing the Fr\'echet derivative of $\Psi$}

\begin{lemma}\label{Frechet diffbility of Psi}
The map $\Psi:\cF\times\cG\to \cH$ is Fr\'echet differentiable in a neighborhood of $(0,0)\in\cF\times\cG$.
\end{lemma}
The proof of Lemma \ref{Frechet diffbility of Psi} is quite technical and will be postponed to the Appendix.

\begin{theorem}\label{Psi'=something}
The Fr\'echet derivative $\Psi'(0,0)$ defines a mapping from $\cF\times\cG$ to $\cH$ by the formula
\begin{equation*}
\Psi'(0,0)[f_0,g_0] = \pa_n u' + \pa_{nn}u \, g_0,
\end{equation*}
where $\pa_{nn}u = n\cdot \left(D^2 u\; n\right)$.
In particular, following the definition of $u'_\pm$ given right after \eqref{bernu u' eq}, we have the following expression for the partial Fr\'echet derivatives as well:
\begin{align}
\pa_f \Psi(0,0)[f_0] &= \pa_n u'_-, \label{deri f} \\
\pa_g \Psi(0,0)[g_0] &= \pa_n u'_+ + \pa_{nn} u\, g_0. \label{deri g} 
\end{align}
\end{theorem}
\begin{proof}
Fix $(f_0,g_0)\in \cF\times\cG$. For simplicity, set $u_t:=u_{tf_0,tg_0}$ and $n_t:=n_{tg_0}$. Since $\Psi$ is Fr\'echet differentiable by Lemma \ref{Frechet diffbility of Psi}, we can compute its Fr\'echet derivative as the following G\^ateaux derivative:
\begin{equation*}
\Psi'(0,0)[f_0,g_0]=\restr{\frac{d}{dt}}{t=0}\Psi(tf_0,tg_0)= \restr{\frac{d}{dt}}{t=0} \Pi_0\Big( \left( \gr u_t\cdot n_t\right)\circ \left( {\rm Id}+tg_0n \right)   \Big).
\end{equation*}
Now, since the projection operator $\Pi_0$ commutes with differentiation, we have
\begin{equation}\label{Psi'=}
\Psi'(0,0)[f_0,g_0]= \Pi_0\left( \restr{\frac{d}{dt}}{t=0} \Big( \left(\gr u_t\cdot n_t \right)\circ \left({\rm Id}+ tg_0n\right)\Big)\right).  
\end{equation}
By \eqref{function of value rewritten}, $\gr u_t \cdot n_t = -|\gr u_t|<0$ on $\pa\Omega_0$. Therefore, we can write
\begin{eqnarray}\nonumber
\Psi'(0,0)[f_0,g_0]= -\Pi_0\left( \restr{\frac{d}{dt}}{t=0} |\gr u_t|\circ \left({\rm Id}+ tg_0n\right)\right) = \\
-\Pi_0 \left(\frac{1}{|\gr u|} \Big(\gr u\cdot \gr u' + (D^2 u\, \gr u)\cdot g_0n  \Big)\right) 
= \Pi_0\left(\pa_n u'\right) +\pa_{nn} u\, \Pi_0g_0,
\end{eqnarray}
where in the last equality we used the fact that $n=-\gr u /|\gr u|$ and that $\pa_{nn} u$ is constant on $\pa\Om_0$. Now, since both maps $f\mapsto \restr{\pa_n u'_-}{\pa\Om_0}$ and $g\mapsto \restr{\pa_n u'_+}{\pa\Om_0}$ preserve the eigenspaces of the Laplace Beltrami operator in the sense of Proposition \ref{prop expansion} and that $\Pi_0 f_0=f_0$ and $\Pi_0 g_0=g_0$ by construction, we obtain that 
\begin{equation*}
    \Psi'(0,0)[f_0,g_0]= 
 \Pi_0\left(\pa_n u'_-\right) + \Pi_0\left( \pa_n u'_+\right) +\pa_{nn} u\, \Pi_0g_0 = \pa_n u'_- + \pa_n u'_+ +\pa_{nn} u\, g_0
\end{equation*}
as claimed. 
The representation formulas for the partial Fr\'echet derivatives $\pa_f \Psi(0,0)$ and $\pa_g \Psi(0,0)$ follow immediately by the definitions of $u'_-$ and $u'_+$.
\end{proof}

Combining Propositions \ref{prop expansion} and \ref{Psi'=something} yields the following. 

\begin{corollary}\label{partial derivatives explicit}
Assume \eqref{h_in h_out exp}. Then the following hold true under the convention \eqref{convention}. 
\begin{equation*}
\begin{aligned} 
\pa_f \Psi(0,0)[f_0]=& 
\displaystyle\frac{2-N}{R^{2-N}-1}R^{2-N}  \displaystyle \sum_{k=0}^\infty \sum_{i=1}^{d_k} \frac{N-2+2k}{R^{2-N-k}-R^k} \al_{k,i}^- Y_{k,i}(\theta),\\     
\pa_g \Psi(0,0)[g_0]=& 
\displaystyle\frac{2-N}{R^{2-N}-1}\sum_{k=1}^\infty \displaystyle \sum_{i=1}^{d_k} \frac{(1-k)R^k+(1-N-k)R^{2-N-k}}{R^{2-N-k}-R^k} \al_{k,i}^+ Y_{k,i}(\theta).     
\end{aligned}
\end{equation*}
\end{corollary}
\subsection{Applying the implicit function theorem}\label{ssection applying ift}
\begin{proof}[Proof of Theorem \ref{main thm 1}]
In what follows, let us assume the result of Lemma \ref{Frechet diffbility of Psi} (see the Appendix for a proof). In order to apply the implicit function theorem (Theorem \ref{ift} of page \pageref{ift}) to $\Psi$, we just need to ensure that the mapping \eqref{deri g} (or, equivalently, the one defined by the second formula of Corollary \ref{partial derivatives explicit}) is a bounded invertible linear transformation from $\cG$ to $\cH$. Linearity and boundedness ensue from the properties of the boundary value problem \eqref{bernu u' eq}. We are left to show that $\pa_g\Phi(0,0):\cG\to\cH$ is a bijection. 
First of all, by Corollary \ref{partial derivatives explicit}, we know that $\pa_g\Psi(0,0):\cG\to\cH$ is given by the map
\begin{equation}\label{another psi_g}
\sum_{k=1}^\infty \sum_{i=0}^{d_k} \al_{k,i}^+ Y_{k,i}\longmapsto \sum_{k=1}^\infty \sum_{i=0}^{d_k} \beta_k \al_{k,i}^+ Y_{k,i},     
\end{equation}
where $\beta_k$ is defined by
\begin{equation}\label{beta_k}
\beta_k= \frac{2-N}{R^{2-N}-1} \ \frac{(1-k)R^k+(1-N-k)R^{2-N-k}}{R^{2-N-k}-R^k}, 
\end{equation}
under the convention \eqref{convention}.
Now, the injectivity of the map \eqref{another psi_g} is an immediate consequence of the fact that, for all $k\ge 1$, the coefficient $\be_k$ in the above never vanishes. Let us now show surjectivity. Take an arbitrary function $h_0\in\cH$. Since, in particular, $h_0$ is continuous on $\pa\Om_0$, it admits a spherical harmonic expansion, say 
\begin{equation*}
    h_0=\sum_{k=1}^\infty\sum_{i=0}^{d_k} \ga_{k,i} Y_{k,i}.
\end{equation*}
Set now 
\begin{equation}\label{g}
    g_0=\sum_{k=1}^\infty\sum_{i=0}^{d_k} \frac{\ga_{k,i}}{\be_k} Y_{k,i}. 
\end{equation}
First of all, notice that, since the sequence $1/\be_k$ is bounded, the function $g_0$ above is a well defined element of $L^2(\pa\Om_0)$. Moreover, the integral of $g_0$ over $\pa\Om_0$ vanishes because the summation in \eqref{g} starts from $k=1$. 
Finally, if we let $\mathcal L$ denote the continuous extension to $L^2(\pa\Om_0)\to L^2(\pa\Om_0)$ of the map defined by \eqref{another psi_g}, it is clear that $g_0=\mathcal{L}^{-1}(h_0)$ by construction. Therefore, in order to prove the surjectivity of the original map $\pa_g\Psi(0,0):\cG\to\cH$, we just need to show that the function $g_0$, defined by \eqref{g}, is of class $\cC^{m,\al}$ whenever $h_0\in \cC^{m-1,\al}$. To this end, we will proceed as in the proof of \cite[Proposition 5.2]{kamburov sciaraffia}. 
First of all, we recall that functions in the Sobolev space $H^s(\pa\Om_0)$ can be characterized by the decay of the coefficients of their spherical harmonic expansion as follows:
\begin{equation*}
    \sum_{k=1}^\infty\sum_{i=0}^{d_k}(1+k^2)^s \al_{k,i}^2 < \infty\quad  \iff\quad \sum_{k=1}^\infty\sum_{i=0}^{d_k} \al_{k,i}^2 Y_{k,i}\in H^s(\pa\Om_0).
\end{equation*}
Since, in particular, $h_0\in\cC^{m-1,\al}(\pa\Om_0)\subset H^{m-1}(\pa\Om_0)$, the asymptotic behavior of the coefficients $\be_k$ given in \eqref{beta_k} yields that $g_0\in H^m(\pa\Om_0)$. Now, let $u'_+$ denote the solution to \eqref{bernu u' eq} where $f_0=0$ and $g_0$ is given by \eqref{g}. By construction, $u_+'$ also satisfies:
\begin{equation}\label{bernu u' eq also satisfies}
\begin{cases}
-\De u_+' = 0 \quad \text{in }\Om_0\setminus\ol D_0,\\
u_+'=0 \quad \text{on } \pa D_0,\\
\pa_n u_+'= h_0-\pa_{nn} u\ g_0  \quad \text{on } \pa \Om_0.
\end{cases}
\end{equation}
Notice that, by assumption, the Neumann data in \eqref{bernu u' eq also satisfies} belongs to $\cC^{m-1,\al}(\pa\Om_0)+H^m(\pa\Om_0)$. Therefore, by elliptic regularity for the Neumann problem, it must be that
\begin{equation*}
u_+' \in \cC^{m,\al}(\ol\Om_0) + H^{m+1}(\Om).
\end{equation*}
Let us argue by induction that
\begin{equation*}
u_+' \in \cC^{m,\al}(\ol \Om_0) + H^{k/2}(\Om_0) \text{ for all } k\ge 2m+2.
\end{equation*}
Indeed, from the inductive assumption we see that the trace of $u_+'$ satisfies
\begin{equation*}
\restr{u_+'}{\pa\Om_0} \in \cC^{m,\al}(\pa\Om_0) + H^{(k-1)/2}(\pa\Om_0),
\end{equation*}
which, in turn, implies that the Neumann data in \eqref{bernu u' eq also satisfies} is in $\cC^{m-1,\al}(\pa\Om_0) + H^{(k-1)/2} (\pa\Om_0)$. Hence, by elliptic regularity for the Neumann problem, $u_+' \in \cC^{m,\al}(\ol \Om_0) + H^{(k+1)/2}(\Om_0)$, which completes the inductive step. By Sobolev
embedding, we now conclude that $u_+' \in \cC^{m,\al}(\ol\Om_0)$, so that its trace $\restr{u_+'}{\pa\Om_0}= h_0-\pa_{nn}u\ g_0$ belongs to $\cC^{m,\al}(\pa\Om_0)$.
In particular, this implies that $g_0\in\cC^{m,\al}(\pa\Om_0)$, as claimed. This concludes the proof of the invertibility of the map $\pa_g\Psi(0,0):\cG\to\cH$. Finally, boundedness ensues by the Schauder boundary estimates and thus the proof of Theorem \ref{main thm 1} is complete.
\end{proof}

Moreover, item $(iii)$ of Theorem \ref{ift} yields the following asymptotic behavior of $g(f)$. 
\begin{corollary}\label{asymptotic thm 1}
Suppose that $f\cdot n=\sum_{k=1}^\infty\sum_{i=1}^{d_k} \al_{k,i}^- Y_{k,i}$ on $\pa D_0$, then
\begin{equation*}
g\left( f  \right)=
\sum_{k=1}^\infty\sum_{i=1}^{d_k} \frac{(2-N-2k)R^{2-N}}{(1-k)R^k+(1-N-k)R^{2-N-k}}\al_{k,i}^- Y_{k,i} + o\left(\norm{f}_{\cC^{0,1}}\right) \quad \text{as } \norm{f}_{\cC^{0,1}}\to0.
\end{equation*}
\end{corollary}

\section{Proof of Theorem \ref{main thm 2}}
The proof of Theorem \ref{main thm 2} follows along the same lines as that of Theorem \ref{main thm 1}, with some obvious modification. 
First of all, fix $\sg_c\ne 1$ and let $\Psi:\cF\times\cG\to\cH$ denote the function defined by \eqref{Psi} in the sense of \eqref{function of value}, with $u_{f,g}$ being the solution to the boundary value problem
\begin{equation}\label{serrin u_fg eq}
\begin{cases}
-\dv(\sg_{f,g} \gr u)=1 \quad \tin \Om_g,\\
u=0\quad  \ton\pa\Om_g,
\end{cases}    
\end{equation}
where $\sg_{f,g}$ denotes the piece-wise constant function defined by \eqref{sigma} with respect to the pair $(D_f,\Om_g)$.
Clearly, $\Psi(f,g)=0$ if and only if the pair $(D_f,\Om_g)$ solves Problem 2. 

In what follows we will admit the shape differentiability of $u_{f,g}$ and the Fr\'echet differentiability of $\Psi$, while their proofs will be postponed to the Appendix.  
The actual explicit expressions for the shape derivative of $u_{f,g}$ and the Fr\'echet derivative of $\Psi$ can be obtained by following the proofs of \cite[Proposition 3.1,  Proposition 3.2]{cava2018} and \cite[Theorem 3.3]{CY1} verbatim. 
As a result, we get the following expressions for the partial Fr\'echet derivatives of $\Psi$ under \eqref{h_in h_out exp}. 

Assume \eqref{h_in h_out exp}. Then the following hold true. 
\begin{equation*}
\begin{aligned} 
\pa_f \Psi(0,0)[f_0]=& 
 \displaystyle \sum_{k=0}^\infty \sum_{i=1}^{d_k}  \frac{2-N-2k}{F}(\sg_c-1)k R^{1-k} \al_{k,i}^- Y_{k,i}(\theta),\\    
\pa_g \Psi(0,0)[g_0]=& 
\sum_{k=1}^\infty \displaystyle \sum_{i=1}^{d_k} 
\frac{(N+k-1)(\sg_c-1)k+(N-2+k+k\sg_c)(k-1)R^{2-N-2k}}{F}\al_{k,i}^+ Y_{k,i}(\theta), 
\end{aligned}
\end{equation*}
where $F= N(N-2+k+k\sg_c)R^{2-N-2k}+k N (1-\sg_c)$.
Now, we can proceed as in Subsection \ref{ssection applying ift} and show that the mapping $\pa_g \Psi(0,0):\cG\to\cH$ is a bounded bijection if and only if $\sg_c\notin \Si$. This completes the proof of Theorem \ref{main thm 2}. Furthermore, by item $(iii)$ of Theorem \ref{ift} we get the following.
\begin{corollary}\label{asymptotic thm 2}
Suppose that $f\cdot n= \sum_{k=1}^\infty \sum_{i=1}^{d_k} \al_{k,i}^- Y_{k,i}$ on $\pa D_0$. Then the following asymptotic behavior holds true as $\norm{f}_{\cC^{0,1}}\to0$.
\begin{equation*}
g(f)=\sum_{k=1}^\infty\sum_{i=1}^{d_k} \frac{\alpha_{k,i}^-(N+2k-2)(\sg_c-1)k R^{1-k}}{(N+k-1)(\sg_c-1)k + (N-2+k+k\sg_c)(k-1)R^{2-N-2k}}  Y_{k,i} + o(\norm{f}_{\cC^{0,1}}).
\end{equation*}
\end{corollary}

\section{Concluding remarks}
\subsection{On the optimal regularity of $\pa D_f$ 
}
To the best of our knowledge, there is no theory of shape derivatives that deals with \emph{perturbation fields} that are not at least Lipschitz continuous. Indeed, it is worth noticing that many known results aim to generalize shape calculus in the other direction: that is, trying to define shape derivatives for non-regular \emph{sets} (even just measurable sets) but with respect to ``regular" perturbation fields (we refer the interested reader to \cite{SG, HP2018} and the references therein). So in this sense, we can state that our result yields the optimal (known) regularity for the fixed boundary $\pa D_f$.  
 
Finally, we would like to remark that applying only the (classical) theory of shape derivatives with respect to Lipschitz continuous perturbations to both boundaries $\pa D_0$ and $\pa \Om_0$ would not have been enough to show the existence of solutions as done in Theorems \ref{main thm 1} and \ref{main thm 2}. Indeed, the functional $\Psi(f,g)$ itself turns out to be not well-defined if the function $g$ is just Lipschitz continuous (since, for instance, one cannot define the trace of $\gr u_{f,g}$ on $\pa\Om_g$ if $\pa\Om_g$ is not Lipschitz continuous).     
 
\subsection{On the optimal regularity of $\pa\Om_g$}
It can be shown (see \cite[Theorem 2]{KN77}) that the free boundary $\pa\Om_g$ given by Theorems \ref{main thm 1} and \ref{main thm 2} is indeed an analytic surface. Despite that, to our knowledge, it is not clear how this result could be obtained directly by the SAP method since the class of analytic functions is not naturally endowed with a Banach space structure.   

In \cite{cava2020}, the author considered a variation of Problem 2 where the overdetermined condition on the normal derivative is been replaced by $\pa_n u(x) = cH(x)$ instead (here $H(x)$ denotes the mean curvature of $\pa\Om$ at the point $x$).
We remark that the SAP method can be applied in this case as well. As a consequence, one can prove existence of solutions of the form $(D_f, \Om_g)$ where $\pa D_f$ is Lipschitz continuous and $\pa \Om_g$ is of class $\cC^{m,\al}$ for arbitrarily large $m$. Notice that, unlike Problem 2, the machinery of \cite{KN77} cannot be applied to obtain the analyticity of the solutions, since the overdetermined condition $\pa_n u(x) = cH(x)$ is not one of the types considered in \cite{KN77}.   

\subsection{Local symmetry of solutions}
Let $(D_f,\Om_g)$ be a solution of Problem 1 (resp. Problem 2) for small enough $(f,g)$ under the hypotheses of Theorem \ref{main thm 1} (resp. Theorem \ref{main thm 2}, in particular, assume $\sg_c\notin\Si$). Then the functions $f$ and $g$ ``share the same symmetries".  This can be made precise by the following:
\begin{proposition}
Let $\ga$ be an element of the orthogonal group $O(N)$ such that $f$ is invariant with respect to $\ga$ (i.e. $f\circ\ga=f$). Then, $g$ is also invariant with respect to $\ga$.
\end{proposition}
\begin{proof}
By assumption, $(D_f,\Om_g)$ is a solution of Problem 1 (resp. Problem 2). Now, since $\ga$ is a rigid motion, $(D_{f\circ\ga}, \Om_{g\circ\ga})$ is also a solution. Furthermore, $D_f=D_{f\circ \ga}$ by hypothesis. Since, by assumption $f$ and $g$ are small enough and $\sg_c\notin\Si$, then we can apply item $(ii)$ of Theorem \ref{ift} to conclude that $g=g\circ \ga=g(f)$ as claimed. 
\end{proof}
We remark that, for Problem 2, this result holds only for sufficiently small $(f,g)$ and $\sg_c\notin\Si$. Indeed, as shown in \cite{CYisaac}, Problem 2 admits solutions of the form $(D_0,\Om_g)$ for $g\ne 0$ branching from the bifurcation point $\sg_c=s_k\in\Si$. In this case, $g$ is invariant with respect to a strictly smaller (and non empty) subset of $O(N)$. 

\subsection{A counterexample concerning convexity}\label{subs convexity}

Let $(D,\Om)$ be a solution of \eqref{bernu}. It is known that, if $D$ is convex then $\Om$ must be as well (see \cite{HS97}). 
In what follows we will show that the converse does not hold. 
To this end, we will make use of the SAP method and construct a counterexample showing that the convexity of $\Om$ does not necessarily imply that of $D$.  
Let $(\boldsymbol 0,R)\in\RR^{N-1}\times \RR$ be the north pole of $\pa D_0$ and let $f_\vee:\RR^{N-1}\times \RR\to\RR^{N-1}\times \RR$ be the following Lipschitz continuous function:
\begin{equation*}
f_\vee(x,y)= \begin{cases}
(\boldsymbol 0,|x|-\varepsilon)\quad & \tfor (x,y)\in \RR^{N-1}\times \RR, \; |x|\le \ve, \; |y-R|\le \ve,\\
(\boldsymbol 0,0)\quad & \text{otherwise}.
\end{cases}    
\end{equation*}
Here $\varepsilon$ is a positive parameter to be chosen such that $\supp f_\vee \subset B$.
Let $(D_t,\Om_t)$ denote the solution of \eqref{bernu} given by Theorem \ref{main thm 1} for $f=t f_\vee$ (see Figure \ref{indentedring}).
\begin{figure}[h]
\centering
\includegraphics[width=0.4\linewidth]{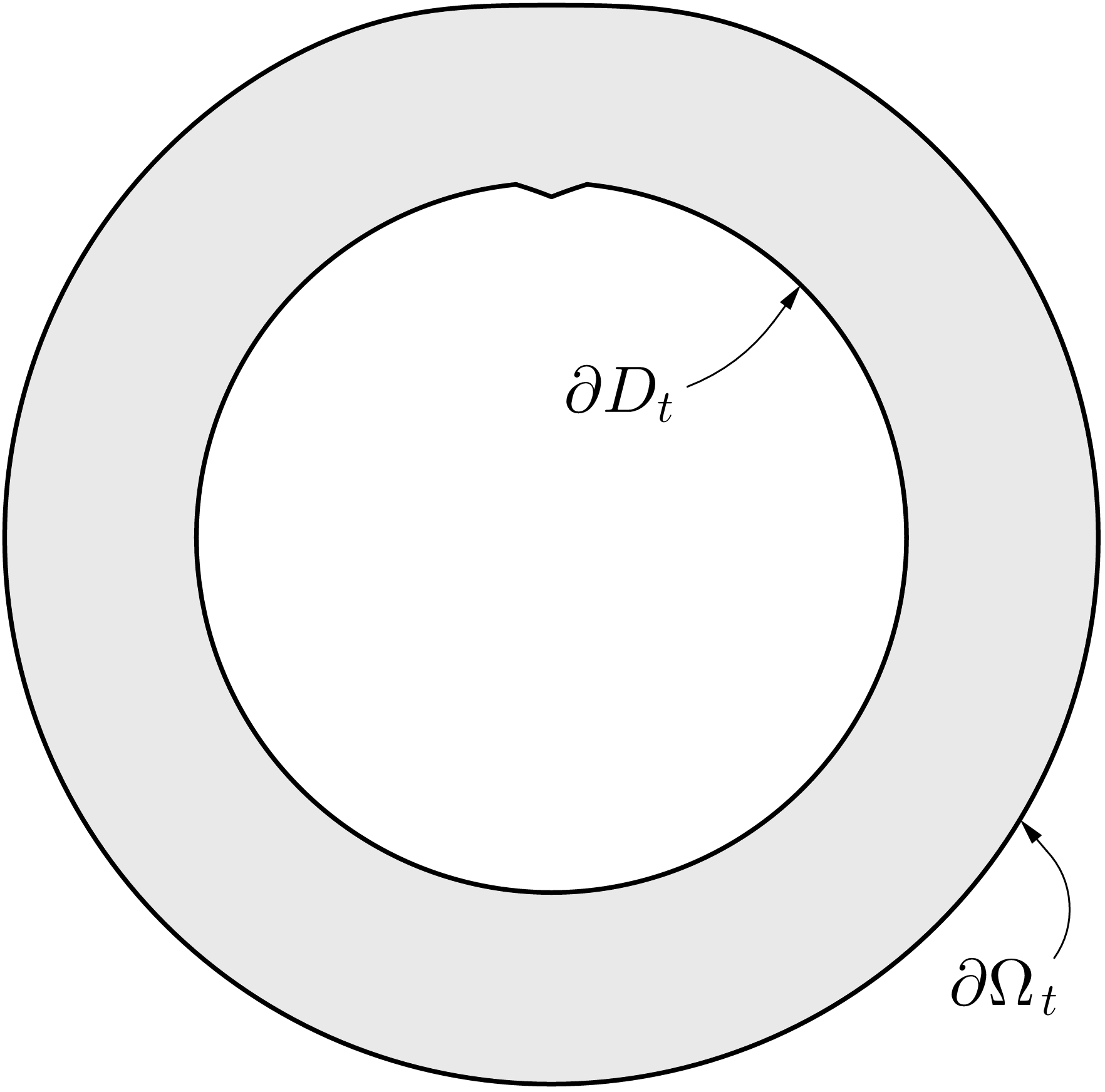}
\caption{Counterexample: $\Om_t$ is convex but $D_t$ is not.} 
\label{indentedring}
\end{figure}
Notice that, by taking $t$ sufficiently small, $\pa\Om_t$ can be made arbitrarily close to $\pa\Om_0$ in the $\cC^2$ norm. As a consequence, we can find some $t>0$ such that $\pa\Om_t$ has positive sectional curvature everywhere and thus is a convex set by \cite{sacksteder}. On the other hand, by construction, the set $D_t$ is never convex, no matter how small $t$ is. 

Clearly, the same considerations can be made for Problem 2 by applying Theorem~\ref{main thm 2}.

\section{Appendix: proof of the Fr\'echet differentiability of $\Psi$}
In this section, we give a proof of the Fr\'echet differentiability of the map $\Psi$. The procedures used are standard (see for instance \cite{SG, HP2018}) but some technicalities arise when we try to ``glue together" perturbations with different regularities.
\subsection{For the Problem 1}

Let $\Om\subset\rn$ be a bounded domain of class $\cC^{m+1,\al}$ and let $D\subset\ol D\subset \Om$ be an open set with Lipschitz continuous boundary. Suppose that $D$ has ``no holes" so that $\Om\setminus\ol D$ is a domain. 
Moreover, for some sufficiently small constant $\de>0$, set 
\begin{eqnarray}\label{K K' B}
\quad K:=\setbld{x\in \ol\Om}{\dist(x,\pa\Om)\le \de}, \quad K':= \setbld{x\in \ol\Om}{\dist(x,\pa\Om)\le 2\de},\\
B:=\setbld{x\in\Om}{\dist(x,\ol D)<\de},
\end{eqnarray}
where $\dist(x,\pa\Om)$ denotes the distance between the point $x$ and the set $\pa\Om$. Notice that, by taking $\de$ small enough, we can assume that $K'\cap\ol D =\emptyset$. 
\begin{figure}[h]
\centering
\includegraphics[width=0.5\linewidth]{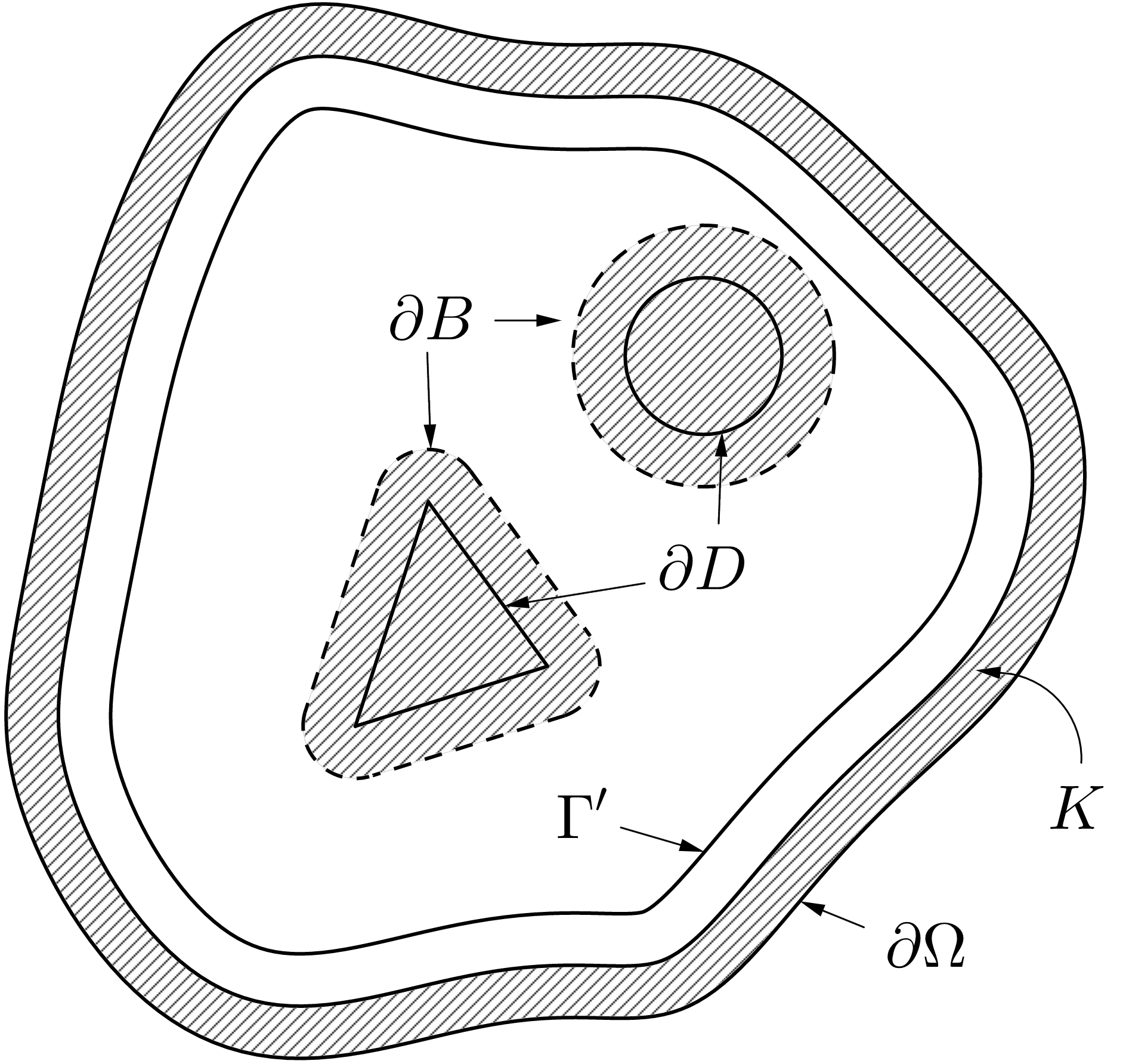}
\caption{The geometrical construction used in this section.} 
\label{nonlinear}
\end{figure}
Finally, define the following Banach space
\begin{equation}\label{A}
\cA:=\setbld{\varphi\in\cC^{0,1}(\ol\Om,\rn)}{
\varphi\equiv 0 \quad \tin \Om\setminus(K^\circ\cup B), \quad 
\restr{\varphi}{K'}\in\cC^{m,\al}(K',\rn)}
\end{equation}
endowed with the norm 
\begin{equation*}
    \norm{\varphi}:=\norm{\varphi}_{W^{1,\infty}(\Om,\rn)}+\norm{\varphi}_{\cC^{m,\al}(K',\rn)}.
\end{equation*}
Notice that, for sufficiently small $\varphi\in\cA$, 
the sets  
\begin{equation}\label{Om_phi}
D_\varphi:=({\rm Id}+\varphi)(D), \quad 
\Om_\varphi:=  ({\rm Id}+\varphi)(\Om)
\end{equation}
are simply connected domains with boundaries of class $\cC^{0,1}$ and $\cC^{m,\al}$ respectively.
Furthermore, notice that the points in $\Om\setminus(K^\circ \cup B)$ are not moved by $\id+\varphi$.
Now, let $u_\varphi$ denote the solution to the following boundary value problem.
\begin{equation}\label{bernu2}
    \begin{cases}
    -\De u_\varphi = 0 \quad &\text{in }\Om_\varphi\setminus \ol D_\varphi,\\
    u_\varphi=1 \quad &\text{on }\pa D_\varphi,\\
    u_\varphi=0 \quad &\text{on }\pa \Om_\varphi.
    \end{cases}
\end{equation}
By the standard Schauder theory for elliptic operators (\cite{GT}), $u_\varphi$ belongs to 
\begin{equation*}
    H^1(\Om_\varphi\setminus\ol{D_\varphi})\cap \cC^{m,\al}(\Om_\varphi\cap W)
\end{equation*}
for any arbitrary open neighborhood $W$ of $\pa\Om_\varphi$ that does not intersect $\pa D_\varphi$. 
Finally, notice that, if $\norm{\varphi}$ is small enough, 
the function
\begin{equation}\label{v_varphi}
v_\varphi:=u_\varphi\circ\restr{\left( {\rm Id}+\varphi \right)}{\Om\setminus \ol{D}}    
\end{equation}
is a well defined element of $H^1(\Om\setminus \ol{D})\cap \cC^{m,\al}(K')$.
Then the following holds true.
\begin{lemma}\label{material differentiability bernu}
The map $\varphi\mapsto v_\varphi\in H^1(\Om\setminus \ol{D})\cap \cC^{m,\al}(K')$ is of class $\cC^\infty$ in a neighborhood of $\varphi=0\in \cA$.
\end{lemma}
\begin{proof}
The proof of this Lemma is quite technical but the overall strategy is simple: we just apply Theorem \ref{ift} of page \pageref{ift} to some map $F$ in order to show the smoothness of the auxiliary function $w_\varphi:=v_\varphi-v_0$ in a neighborhood of $\varphi=0\in\cA$ (here $v_0=u_0$ denotes the function $v_\varphi$ corresponding to $\varphi=0$). 

{\bf Step 1: find a functional $F$ such that  $F(\varphi,w_\varphi)=0$.\quad}
First of all, notice that the function $v_\varphi$ is characterized as the unique element of $H^1(\Om\setminus\ol D)$ that satisfies
\begin{equation}\label{v_te eq wk}
\int_{\Om\setminus\ol D} A_\varphi\gr v_\varphi\cdot \gr \psi = 0
\quad\text{ for all }\psi\in H_0^1(\Om), \quad v_\varphi=1 \;\ton\pa D, \quad v_\varphi=0 \;\ton\pa\Om.
\end{equation}
where $J_\varphi$ is the Jacobian of the map $\id+\varphi$ and 
\begin{equation}\label{A li seme}
A_\varphi:=
J_\varphi \left(I+ D\varphi\right)^{-1} ({I}+D\varphi^T)^{-1}.
\end{equation}
This can be proved by explicitly computing the change of variable \eqref{v_varphi} in the weak formulation of $u_\varphi$.
Let now consider $w_\varphi$. By the above,  $w_\varphi$ can be characterized as the unique solution of 
\begin{equation}\label{eq w_varphi}
\int_\Om A_\varphi \gr w_\varphi\cdot\gr\psi +\int_\Om A_\varphi\gr v_0\cdot \gr \psi \quad \text{for all }\psi\in H^1_0(\Om),\quad w_\varphi\in X, 
\end{equation}
where $X$ denotes the Banach space
\begin{equation*}
X:=\setbld{w\in H^1_0(\Om\setminus\ol D)}{\De w \equiv0 \quad \tin \Om\setminus(K\cup \ol B), \quad w\in\cC^{m,\al}(K')},     
\end{equation*}
endowed with the norm $\norm{\cdottone}:=\norm{\cdottone}_{H_0^1(\Om\setminus\ol D)}+\norm{\cdottone}_{\cC^{m,\al}(K')}$.

Let us now consider the following mapping:
\begin{equation}\label{operator F}
F: \cA\times X \ni (\varphi, w) \mapsto -\dv \left(A_\varphi\gr w \right) -\dv \left(A_\varphi \gr v_0 \right) \in  Y,
\end{equation}
where $Y$ denotes the Banach space
\begin{equation*}
Y:=\setbld{h\in H^{-1}(\Om\setminus\ol D)}{h \equiv0 \quad \tin \Om\setminus(K\cup \ol B), \quad w\in\cC^{m-2,\al}(K')},     
\end{equation*}
endowed with the norm $\norm{\cdottone}:=\norm{\cdottone}_{H^{-1}(\Om\setminus\ol D)}+\norm{\cdottone}_{\cC^{m-2,\al}(K')}$.
By \eqref{eq w_varphi}, we have $F(\varphi,w_\varphi)=0$. 

{\bf Step 2: show that $F$ is smooth.\quad}
First, we claim that $F$ is differentiable infinitely many times in a neighborhood of $(0,0)$. As a matter of fact, the map $\cA \ni \varphi\mapsto J_\varphi=\det( { I}+D\varphi)\in L^\infty(\Om)\cap \cC^{m-1,\al}(K')$ is differentiable infinitely many times because also $\varphi\mapsto {I}+D\varphi\in L^\infty(\Om,\RR^{N\times N})\cap \cC^{m-1,\al}(K',\RR^{N\times N})$ is, and the application $\det(\cdottone)$ is a polynomial in its entries and is therefore continuous. 
Similarly, the map $\varphi\mapsto (I+D\varphi)^{-1}$ can be expressed as a Neumann series as $\pp{I+D\varphi}\inv=\sum_{k=0}^\infty(-1)^k (D\phi)^k$ and thus it is $\C^\infty$ in a neighborhood of $0\in \cA$. Therefore, the map $\cA\ni \varphi\mapsto A_\varphi\in L^\infty(\Om,\RR^{N\times N})\cap \cC^{m-1,\al}(K',\RR^{N\times N})$ is also of class $\C^\infty$. Thus, the map $\left(L^\infty(\rn,\RR^{N\times N})\cap \cC^{m-1,\al}(K,\RR^{N\times N})\right)\times X \to H^{-1}(\Om)\cap \cC^{m-2,\al}(K)$ defined by $(A,v) \mapsto -\dv(A \gr v)$ is also of class $\C^\infty$ because both bilinear and continuous. By composition, we conclude that the full map $(\varphi,w)\mapsto F(\varphi,w)$ is of class $\C^\infty$. 

{\bf Step 3: show that $\pa_w F(0,0)$ is a bounded bijection.\quad}
It is easy to see that the partial Fr\'echet derivative with respect to the variable $w$: $\pa_w F(0,0): X\to Y$ is given by the formula $w\mapsto -\De w$. 
In what follows, let us show that the map $X\to Y$ given by $w\mapsto -\De w$ is indeed a bounded bijection as needed by the hypotheses of Theorem \ref{ift}. 
Fix $h\in Y$. Let $w$ denote the unique solution of $-\De w = h$ in $\in H^1_0(\Om\setminus \ol D)$. We will show that $w\in X$. By assumption $w$ is harmonic in $\Om\setminus(K\cup \ol B)$. Take now another compact set $K''$ such that $K\subset K''\subset K'$. By the classical boundary Schauder estimates for the Poisson equation, we get $w\in\cC^{m,\al}(K'')$. Moreover, since $w$ is harmonic (and thus real analytic) in the whole $\Om\setminus(K\cup \ol B)$, in particular $w\in\cC^{m,\al}(K)$ also holds. Since $h\in Y$ was arbitrary, we showed that $\pa_w F(0,0):X\to Y$ is a bijection. Finally, boundedness ensues by the standard regularity theory for the Laplace operator.

{\bf Step 4: apply the implicit function theorem.\quad}
As a consequence of the above, we can apply Theorem \ref{ift} to show the existence of a $\C^\infty$ branch $\phi\mapsto w(\varphi)\in X$ defined for sufficiently small $\varphi\in \cA$ such that $F(\varphi,w(\varphi))=0$. Unique solvability for problem \eqref{eq w_varphi} yields that $w(\varphi)=w_\varphi$. Therefore, we obtain the smoothness of the map $\varphi\mapsto v_\varphi$, as claimed.
\end{proof}
For sufficiently small $\de>0$ let $B$ and $K$ be the sets defined at the beginning of this subsection. Suppose that $\de>0$ is sufficiently small so that $B\cap K=\emptyset$. Now, let $\cF$ and $\cG$ denote the Banach spaces defined in \eqref{FGH}. 
Take any sufficiently small pair $(f,g)$ in $\cF\times\cG$. We will construct a map $\varphi_{f,g}\in\cA$ that verifies \eqref{extending perturbations}. 
Let $\Ga'=\pa K'\setminus \pa\Om_0$ and consider the following two boundary value problems:
\begin{equation}\label{tilde d and g}
\begin{minipage}[h]{0.5\textwidth}
\vspace{0pt}
\begin{equation*}
\begin{cases}
-\De \widetilde d =0\quad &\tin (K')^\circ,\\
\widetilde d = 0 \quad &\ton \Ga',\\
\widetilde d = 1 \quad &\ton \pa\Om.
\end{cases}    
\end{equation*}
\end{minipage}
\hfill
\begin{minipage}[h]{0.4\textwidth}
\vspace{0pt}
\begin{equation*}
 \begin{cases}
-\De \widetilde g =0\quad &\tin (K')^\circ,\\
\widetilde g = 0 \quad&\ton \Ga',\\
\widetilde g = \dfrac{1}{|\gr \widetilde d|}\ g \quad &\ton \pa\Om. 
\end{cases}
\end{equation*}
\end{minipage}    
\end{equation}
Now, for $x\in\ol\Om_0$ set 
\begin{equation*}
\widetilde\varphi_{f,g}(x):=\begin{cases}
f(x)\quad &x\in\ol B,\\
0 \quad & x\in\ol\Om\setminus(\ol B\cup K' ),\\
\widetilde g(x) \gr \widetilde d(x) \quad & x\in K'.
\end{cases}    
\end{equation*}
Let now $\xi\in\cC^\infty([0,\infty), [0,1])$ be a cut off function that verifies 
\begin{equation*}
    \xi \equiv 1 \quad \tin \left[0, \tfrac1 3 \de\right],
    \quad 
    \xi \equiv 0 \quad \tin \left[\tfrac 2 3 \de, \infty\right)
\end{equation*}
and set 
\begin{equation}\label{varphi_fg nuovo}
    \varphi_{f,g}(x):= \widetilde\varphi_{f,g}(x) \; \xi \left( \dist(x,\pa\Om) \right).
\end{equation}
By construction, $\varphi_{f,g}$ is a well-defined element of $\cA$ that satisfies \eqref{extending perturbations} (see also \cite{distance function}). Let now $v_{f,g}$ denote the function \eqref{v_varphi} corresponding to $\varphi=\varphi_{f,g}$.

\begin{lemma}[Fr\'echet differentiability of $v_{f,g}$]\label{frechet diffbility of v_fg}
The map $\cF\times\cG\ni(f,g)\mapsto v_{f,g}\in X$ is Fr\'echet differentiable in a neighborhood of $(f,g)=(0,0)$. 
\end{lemma}
\begin{proof}
First of all, by the linearity of the two boundary value problems in \eqref{tilde d and g}, we deduce that the map $(f,g)\mapsto \varphi_{f,g}\in \cA$ is bilinear. Moreover, by the standard Schauder estimates for the Laplace equation, there exists some constant $C>0$ (independent of $f$ and $g$) such that 
\begin{equation*}
    \norm{\varphi_{f,g}}\le C\left(\  \norm{f}_{W^{1,\infty}(B,\rn)} + \norm{g}_{\cC^{m,\al}(\pa\Om_0,\RR)}\ \right).
\end{equation*}
That is, $(f,g)\mapsto \varphi_{f,g}$ is a continuous bilinear map from $\cF\times\cG$ to $\cA$ , hence it is Fr\'echet differentiable. The claim now follows from Lemma \ref{material differentiability bernu} by composition. 
\end{proof}
\begin{remark}\label{shape differentiability u_fg}
By \eqref{f' by material deri}, the result above implies the shape differentiability of $u_{f,g}$ at $(0,0)$.
\end{remark}
We finally have all the ingredients to prove Lemma \ref{Frechet diffbility of Psi}.
\begin{proof}[Proof of Lemma \ref{Frechet diffbility of Psi}]
Notice that, by change of variables, we have
\begin{equation}\label{change of variables}
\gr u_{f,g}\circ\left(\id + gn\right)^{-1} =
\left(\id+D\varphi_{f,g}\right)^{-1}\gr v_{f,g} \quad \ton \pa\Om_0
\end{equation}
where $\varphi_{f,g}$ is the map defined by \eqref{varphi_fg nuovo}.
The statement of Lemma \ref{Frechet diffbility of Psi} follows by combining \eqref{Psi}, \eqref{function of value rewritten}, \eqref{change of variables} and Lemma \ref{frechet diffbility of v_fg}. 
\end{proof}

\subsection{For Problem 2}
As done in the previous subsection, the result will be given under a fairly general geometrical setting. The proofs will be omitted altogether since they follow almost verbatim from those in the previous subsection. 
Let $\Om\subset\rn$ be a bounded domain of class $\cC^{m+1,\al}$ and let $D\subset\ol D\subset \Om$ be a measurable set. For small enough $\de>0$, define the sets $K$, $K'$ and $B$ as in \eqref{K K' B} and define $\cA$ as in \eqref{A}.
For $\varphi\in\cA$ small enough, let $D_\varphi$ and $\Om_\varphi$ be the sets defined in \eqref{Om_phi}. Moreover, for $\varphi\in\cA$ small enough, let $u_\varphi$ denote the solution to the boundary value problem
\begin{equation}\label{serrin 2}
    \begin{cases}
    -\dv\pp{\sg_\varphi \gr u_\varphi} = 1 \quad &\text{in }\Om_\varphi,\\
    u_\varphi=0 \quad &\text{on }\pa \Om_\varphi,
    \end{cases}
\end{equation}
with $\sg_\varphi:= \sg_c\cX_{D_\varphi}+\cX_{\Om_\varphi\setminus D_\varphi}$.

By the classical regularity theory for elliptic operators in divergence form and the Schauder boundary estimates for the Laplace operator (\cite{GT}), $u_\varphi$ belongs to 
\begin{equation*}
    H^1_0(\Om_\varphi)\cap \cC^{m,\al}(\Om_\varphi\cap W),
\end{equation*}
where $W$ is any open neighborhood of $\pa\Om_\varphi$ that does not intersect $\pa D_\varphi$. 
Moreover, if $\norm{\varphi}$ is small enough, 
the function
\begin{equation*}
v_\varphi:=u_\varphi\circ\restr{\left( {\rm Id}+\varphi \right)}{\Om}    
\end{equation*}
is a well defined element of $H_0^1(\Om)\cap \cC^{m,\al}(K')$. 
As done in Lemma \ref{material differentiability bernu}, we can obtain the smoothness of $w_\varphi:=v_\varphi-v_0$ in the $X$-norm in a neighborhood of $\varphi=0\in\cA$. The shape differentiability of $u_\varphi$ and the Fr\'echet differentiability of $\Psi$ (defined as in \eqref{Psi}) then follow immediately. 

\section*{Acknowledgements}
The author would like to thank Prof. Antoine Henrot (Institut Elie Cartan de Lorraine and Universit\'e de Lorraine) for a fruitful discussion concerning the perturbation of convex sets discussed in Subsection \ref{subs convexity}.

\begin{small}

\end{small}

\bigskip

\noindent
\textsc{
Mathematical Institute, Tohoku University, Aoba, 
Sendai 980-8578, Japan } \\
\noindent
{\em Electronic mail address:}
cavallina.lorenzo.e6@tohoku.ac.jp

\end{document}